\documentclass[a4paper, 10pt, DIV10, headinclude=false, footinclude=false]{scrartcl}

\usepackage[utf8]{inputenc}
\usepackage[T1]{fontenc}
\usepackage[english]{babel}

\usepackage[pdftex]{graphicx}
\usepackage{subcaption}
\usepackage{latexsym}
\usepackage{amsmath,amssymb,amsthm}
\usepackage{esint} 
\usepackage{aligned-overset}
\usepackage{multirow}
\usepackage{bigdelim}
\usepackage{empheq} 
\usepackage{tikz}
\usetikzlibrary{positioning,shapes.multipart}

\newtheorem{Theorem}{Theorem}[section]
\newtheorem{Lemma}[Theorem]{Lemma}	
\newtheorem{Proposition}[Theorem]{Proposition} 
\theoremstyle{definition}
 
\newtheorem{Assumption}[Theorem]{Assumption}
\theoremstyle{remark}	  
\newtheorem{Remark}[Theorem]{Remark}                  
\numberwithin{equation}{section} 
\numberwithin{figure}{section}

\newenvironment{abstr}[1]{ \vspace{.05in}\footnotesize
	\parindent .2in
	{\upshape\bfseries #1. }\ignorespaces}{\par\vspace{.1in}}
\newenvironment{Abstract}{\begin{abstr}{Abstract}}{\end{abstr}}
\newenvironment{keywords}{\begin{abstr}{Key words}}{\end{abstr}}
\newenvironment{AMS}{\begin{abstr}{AMS subject classifications}}{\end{abstr}}

\allowdisplaybreaks[4]

\begin{document}
	\title{Wave propagation in high-contrast media: periodic and beyond%
		\thanks{Funded by the Deutsche Forschungsgemeinschaft (DFG, German Research Foundation) under VE 1397/2-1. Major parts of this work were accomplished while BV was affiliated with Karlsruher Institut für Technologie (KIT) and EF conducted a research internship at KIT. The work of BV at University Bonn is also funded by the Deutsche Forschungsgemeinschaft (DFG, German Research Foundation) under Germany’s Excellence Strategy – EXC-2047/1 – 390685813.}
	}
	\author{\'Elise Fressart\footnotemark[2] \and Barbara Verf\"urth\footnotemark[3]}
	\date{}
	\maketitle
	
	\renewcommand{\thefootnote}{\fnsymbol{footnote}}
	\footnotetext[3]{Institut für Numerische Simulation, Universit\"at Bonn, Friedrich-Hirzebruch-Allee 7, D-53115 Bonn, Germany}
	\footnotetext[2]{\'Ecole Nationale des Ponts et Chauss\'ees (ENPC), 6 et 8 Avenue Blaise Pascal, 77455 Marne-la-Vall\'ee cedex 2, France}
	\renewcommand{\thefootnote}{\arabic{footnote}}
	
	\begin{Abstract} 
	This work is concerned with the classical wave equation with a high-contrast coefficient in the spatial derivative operator. We first treat the periodic case, where we derive a new limit in the one-dimensional case. The behavior is illustrated numerically and contrasted to the higher-dimensional case. For general unstructured high-contrast coefficients, we present the Localized Orthogonal Decomposition	and show a priori error estimates in suitably weighted norms. Numerical experiments illustrate the convergence rates in various settings.
	\end{Abstract}
	
	

\section{Introduction}
The classical wave equation is a popular prototypical model to study the propagation of various types of waves, e.g., in acoustics or electromagnetics. 
In the avenue of modern metamaterials, the study of wave propagation through heterogeneous, or even more precisely, multiscale materials receives a lot of interest. 
In such media, the properties of the material vary on short spatial scales, for instance, because the material is a composite with fine-scale features.
The material properties of the, for simplicity, two components can be of very different magnitude leading to a \emph{high contrast} in the material properties over the whole metamaterial bulk.
Mathematically, this is often modeled by coupling the contrast in the PDE coefficients to the fine-scale parameter $\epsilon$, which originally describes the length scale on which the coefficients variations occur. 

Such a high contrast is known to produce interesting wave propagation phenomena, which also have  been investigated mathematically, often with the help of asymptotic analysis. 
An incomplete list (focusing on mathematical contributions) includes Helmholtz resonators \cite{LaSc2017}, artificial magnetism \cite{BoBoFe2017}, and negative refraction \cite{LaSc2016}. We also refer to the overviews \cite{Sch2017} and \cite{AmDaHiLeYu2022} and references therein.
Most of these works consider \emph{time-harmonic} wave propagation.
Inspired by the analytical results, the numerical homogenization of time-harmonic waves in high-contrast media has been addressed with the Heterogeneous Multiscale Method in \cite{OhVe2018,OhScUrVe2020,Ver2019} and with the Localized Orthogonal Decomposition (LOD) in \cite{PeVe2020}.

On the other hand, the time-domain wave equation in multiscale materials with \emph{low contrast} is also well studied. Asymptotic analysis results date back to \cite{BeLiPa1978}.
Numerical multiscale methods for the wave equation are reviewed in \cite{AbHe2017}. Focusing on unstructured coefficients with a continuum of scales, rough polyharmonic splines were used in \cite{OwZh2008} and the LOD in \cite{AbHe2016}. The LOD has then been combined with explicit time stepping \cite{MaPe2019}, mass lumping \cite{GeMa2021} and time-dependent coefficients \cite{PeVe2020}. We refer to \cite{AlHePe2021,MaPe2021} for further applications of the LOD besides the classical wave equation.

Motivated by these results, we aim to (numerically) study the \emph{time-domain wave equation in a high-contrast material}. Specifically, we assume the coefficient in the spatial derivative of the wave equation to take the values $1$ or $\epsilon^2$.
The exact scaling of the high contrast is motivated by the fact that this choice has turned out to result in unusual limit equations in the time-harmonic \cite{BoBoFe2017} as well as the time-domain setting \cite{NaSi2016}.
Our study can roughly be divided into two parts. In the first part, we consider an $\epsilon$-periodic material. Asymptotic analysis results show a significant difference between one and higher space dimensions. We provide a new limit characterization in the one-dimensional setting, which gives an explicit and simple formula for the homogenized solution in case of vanishing right-hand side. The higher-dimensional setting has been treated in \cite{NaSi2016} in detail and we briefly review the main findings.
Direct numerical simulations illustrate the differences and connections between the low- and high-contrast as well as the one- and the two-dimensional case.
In the second part of our study, we consider a general unstructured high-contrast material. We use the LOD for the wave equation in \cite{AbHe2016} combined with specific interpolation operators designed for stationary high-contrast problems in \cite{HeMa2017,PeSc2016,PeVe2020}. The main contribution consists in a priori error estimates for the semi-discrete method in the high-contrast case. The extension from stationary to time-dependent problems turns out to be non-trivial, in particular because of the necessity to work with weighted spatial norms. We carefully discuss the $\epsilon$-dependence of our bounds.
Several numerical experiments underline the applicability of the approach in one- as well as two-dimensional settings, even for unstructured coefficients.
Overall, we hope to showcase the analytical and numerical implications of wave propagation in high-contrast media and to inspire further research on the topic.

The paper is organized as follows. In Section \ref{sec:setting}, we describe the considered settings in detail. Section \ref{sec:asympotic} contains the first part of our study, namely the periodic setting with asymptotic analysis results and associated numerical illustrations.
The second part of our study concerning the numerical multiscale method beyond the periodic case is contained in Section \ref{sec:numeric}. Finally, we draw some conclusions and point out open questions.

\section{Setting}\label{sec:setting}
Throughout this article, we use standard notation on Lebesgue, Sobolev, and Bochner spaces.  Scalar products are denoted by $(\cdot, \cdot)$ with an index specifying the Hilbert space (mostly $L^2$) and by $\langle \cdot, \cdot\rangle$, we denote the duality paring between $H^{-1}$ and $H^1_0$.
$\Omega\subset \mathbb{R}^d$ denotes the open, connected Lipschitz (spatial) domain and $[0,T]$ the time domain with final time $T$.

We consider the following wave equation
\begin{equation}
\label{eq:wave_eq}
\left\{
\begin{array}{rlll}
\partial_{tt} u_\epsilon - \nabla \cdot (a_\epsilon(x) \nabla u_\epsilon) &=& f & \text{ in \ } \Omega \times (0,T], \\ 
u_\epsilon(x,t) &=& 0 & \text{ on \ } \partial\Omega \times [0,T], \\
u_\epsilon(x,0) &=& u^0(x), & \\
\partial_t u_\epsilon(x,0) &=& v^0(x). & \\
\end{array}
\right.
\end{equation}
with source term $f\in L^2(0,T,L^2(\Omega))$ and initial values $u^0\in H^1_0(\Omega)$ and $v^0\in L^2(\Omega)$. We assume that $a_\epsilon\in L^\infty(\Omega)$ takes only two values, precisely
\begin{equation}\label{eq:aeps}
a_\epsilon=\begin{cases}
a_0\qquad &\text{in } \Omega_\epsilon,\\
1\qquad &\text{else},
\end{cases}
\end{equation}
for some positive constant $a_0>0$.
The subscript $\epsilon$ indicates multiscale features, i.e., we implicitly assume that $\Omega_\epsilon$ consists of $\mathcal O(\epsilon^{-d} )$ connected components with diameter $\mathcal O(\varepsilon)$.
More detailed assumptions on $a_\epsilon$ will be specified in each section. 

It is well known that under the above assumptions, there exists a unique weak solution to \eqref{eq:wave_eq}, i.e., there exists a unique $u_\epsilon \in L^2(0, T, H_0^1(\Omega))$ with $\partial_t u_\epsilon \in L^2(0, T, L^2(\Omega))$ and $ \partial_{tt} u_\epsilon \in L^2(0, T, H^{-1}(\Omega))$ such that 
\begin{equation}
\label{eq:wave_eq_variational_form}
\left\{
\begin{array}{rlll}
\langle \partial_{tt} u_\epsilon,  v\rangle + ( a_\epsilon \nabla u_\epsilon , \nabla v)_{L^2(\Omega)} &=& (f,v)_{L^2(\Omega)} & \text{\hspace{1cm}} \forall v \in H_0^1(\Omega) \text{, for a.e. } t \in [0,T], \\
u_\epsilon(\cdot,0) &=& u^0, & \\
\partial_t u_\epsilon(\cdot,0) &=& v^0. & \\
\end{array}
\right.
\end{equation}
We introduce the notation $b_\epsilon(v,w) = \int_\Omega a_\epsilon \nabla v \cdot \nabla w$ for all $v,w \in H_0^1(\Omega)$.

The goal is to compare the behavior of $u_\epsilon$ for the \emph{low contrast} case with $a_0=\mathcal O(1)$ and the \emph{high contrast} one with $a_0=\mathcal{O}(\epsilon^2)$.
In particular we emphasize that $b_\epsilon$ is almost degenerate in the high contrast case in the sense that its coercivity constant tends to zero for $\epsilon \to 0$.
One could also consider other $\epsilon$-scalings of $a_0$. However, we restrict our studies to this specific setting as it turned out to show interesting resonances in the time-harmonic case.
Let us emphasize at this point that we only consider final times $T=\mathcal O(1)$ and thereby avoid any considerations of additional long-time effects.

\section{Asymptotic results in the periodic case}\label{sec:asympotic}
The periodic case is specified via the (periodic) form of $\Omega_\epsilon$ which we define as
\begin{equation*}
	\Omega_\epsilon=\Bigl(\bigcup_{j\in \mathbb{Z}^d}\epsilon(j+\Sigma)\Bigr)\cap \Omega
\end{equation*}
for a subdomain $\Sigma\subset Y\coloneqq (0,1)^d$.
We call $Y$ the unit cell and $\Sigma$ the inclusion.
For $d=1$, we only assume that $\Sigma$ is a true subdomain, i.e., unequal to $Y$. For $d\geq 2$, we assume that $\Sigma$ is compactly embedded in $Y$, i.e., it does not touch the boundary of $Y$. This prevents high contrast features to run through the whole domain $\Omega$.

In the low contrast setting, the limit $\epsilon\to 0$ is well studied using homogenization theory. We have the following result for the solution $u_\epsilon$ to \eqref{eq:wave_eq_variational_form}, cf.~\cite{AbHe2016,AbHe2017,BeLiPa1978}. It holds
\begin{equation*}
	\begin{array}{rlll}
		u_\epsilon  &\rightharpoonup & \hat{u} & \text{weakly-* in } L^\infty(0,T,H_0^1(\Omega)) \\
		\partial_t u_\epsilon  &\rightharpoonup & \partial_t \hat{u} & \text{weakly-* in } L^\infty(0,T,L^2(\Omega))
	\end{array}
\end{equation*}
where $\hat{u} \in L^2(0, T, H_0^1(\Omega))$, $\partial_t \hat u \in L^2(0, T, L^2(\Omega))$, $ \partial_{tt} \hat u \in L^2(0, T, H^{-1}(\Omega))$ and $\hat{u}$ is the solution of the homogenized problem
\begin{equation}
	\label{eq:wave_eq_hom}
	\left\{
	\begin{array}{rlll}
		\langle \partial_{tt} \hat{u},v \rangle + (\hat{a} \nabla \hat{u}, \nabla v)_{L^2(\Omega)} &=& (f,v)_{L^2(\Omega)} & \forall v \in H_0^1(\Omega) \text{ and a.e. } t\in [0,T] \\ 
		\hat{u}(\cdot,0) &=& u^0 & \\
		\partial_t \hat{u}(\cdot,0) &=& v^0 & \\
	\end{array}
	\right.
\end{equation}
Especially in the periodic case, details on the homogenized coefficient $\hat a$ are available. They are presented depending on the space dimension in Section \ref{subsec:asymptotic:1d} and Section \ref{subsec:asymptotic:higherd}, respectively, where we also discuss results for the high-contrast case. Numerical illustrations in Section \ref{subsec:asymptotic:numexp} complement the theoretical results and showcase the different behavior in one and two space dimensions.

\subsection{One-dimensional case}\label{subsec:asymptotic:1d}
In the one-dimensional case, $\hat a$ in \eqref{eq:wave_eq_hom} takes a simple form, namely the harmonic average, i.e.,
\[\hat a=\Bigl(\fint_0^\epsilon \frac{1}{a_\epsilon}\, dx\Bigr)^{-1}.\]
Using the even more specific definition of $a_\epsilon$ in \eqref{eq:aeps}, we obtain 
\begin{equation}\label{eq:ahat1d}
\hat{a}=\Bigl( \frac{|\Sigma|}{a_0}+1-|\Sigma|\Bigr)^{-1}=\frac{a_0}{a_0+(1-a_0)|\Sigma|},
\end{equation}
where $|\Sigma|$ denotes the length of the subinterval $\Sigma$.

In the high-contrast case, we have the following asymptotic result.
\begin{Proposition}\label{prop:homhighcontrast1d}
	Let $d=1$, $\epsilon\leq 1$ and let $u_\epsilon$ be the solution to \eqref{eq:wave_eq_variational_form} with $a_0=\epsilon^2$, $\Omega_\epsilon$ as specified above, $f\equiv 0$ and initial values $u^0,v^0\in H^1_0(\Omega)$.
	Then there is $\hat u\in L^2(0,T, H^1_0(\Omega))$ with $\partial_t \hat u\in L^2(0,T, L^2(\Omega))$ such that $u_\epsilon$ two-scale converges to $\hat u$ as $\epsilon\to 0$. $\hat u$ is explicitly given as
	\[\hat{u}(x,t)=u^0(x)+v^0(x)t.\]
\end{Proposition}
Note that, formally, we would get $\hat{a}\to 0$ for $\epsilon\to0$ which agrees with the above formula for $\hat u$.
\begin{proof}
	Since $f\equiv0$ we have conservation of the energy
	\[E(t)^2\coloneqq \frac12\|\partial_t u_\epsilon(\cdot, t)\|_{L^2(\Omega)}^2+\frac 12\|\sqrt{a}_\epsilon\nabla u_\epsilon(\cdot, t)\|_{L^2(\Omega)}^2,\]
	i.e., $E(t)=E(0)$ for all $t$.
	By the assumptions on $u^0$, $v^0$ and $a_\epsilon$, $E(0)$ is uniformly bounded in $\epsilon$ for $\epsilon\leq 1$. This implies that 
	\begin{align*}
		\|\partial_t u_\epsilon\|_{L^\infty(0,T, L^2(\Omega))}&<C, \qquad
		&\epsilon\|\nabla u_\epsilon\|_{L^\infty(0,T, L^2(\Omega))}&<C, \qquad &\|\nabla u_\epsilon\|_{L^\infty(0,T, L^2(\Omega\setminus \Omega_\epsilon))}&< C
	\end{align*}
with constants $C$ independent of $\epsilon$. 
Standard energy-type estimates for the wave equation further yield that
\begin{align*}
\|u_\epsilon\|_{L^\infty(0,T, L^2(\Omega))}&<C
\end{align*}
with a constant $C$ only depending on the initial values and thus independent of $\epsilon$.
Note that since $T$ is finite, we get similar estimates as the four above also for $L^2$-norms with respect to time instead of $L^\infty$-norms.

Using the arguments from the stationary case in \cite[Appendix A]{ChChCo2018} and treating $t$ as a parameter, we see that there exists $\overline{u}\in L^2((0,T)\times \Omega, V_1)$ such that --- up to a subsequence ---
\begin{itemize}
	\item $u_\epsilon$ two-scale converges to $\overline u$,
	\item $\epsilon\nabla u_\epsilon$ two-scale converges to $\nabla_y \overline u$,
	\item $\chi_{\Omega\setminus\Omega_\epsilon}\nabla u_\epsilon$ two-scale converges to zero.
\end{itemize}
Here, $\chi_A$ denotes the characteristic function of the set $A$ and we define the space $V_1\coloneqq \{v\in H^1_{\#}(Y): \nabla_y v=0 \text{ in }Y\setminus \overline{\Sigma}\}$, where the subscript $\#$ indicates periodic function spaces. By the boundedness of $\partial_t u_\epsilon$ and the two-scale convergence of $u_\epsilon$, we in fact get, cf. \cite{NaSi2016}, that $\partial_t \overline u \in L^2((0,T)\times \Omega\times Y)$ and $\partial_t u_\epsilon$ two-scale converges to $\partial_t  \overline u$.
 
With these convergences and choosing a test function of the form $v(x,x/\epsilon)$ with $v\in C^\infty_c(\Omega, V_1)$, we can pass to the limit in \eqref{eq:wave_eq_variational_form}. The second time derivative of $u_\epsilon$ and the initial conditions can be treated as discussed in detail in \cite[Sec.~3]{NaSi2016}. We get the limit system
\[\langle\partial_{tt}\overline u, \psi\rangle_{\Omega\times Y}+(\nabla_y \overline{u}, \nabla_y\psi)_{L^2(\Omega\times \Sigma)}=0\qquad\forall \psi\in L^2(\Omega, V_1)\]
with the initial conditions $\overline u(x,y,0)\, dx=u^0(x)$ and $\partial_t \overline u(x,y,0)\, dy=v^0(x)$.
Note that the result so far only holds for a subsequence. But since the solution $\overline u$ is unique by energy arguments, the result then holds for the whole sequence.

The dependence on $x$ in this limit equation is only parametric. In other words, we can write $\overline{u}(x,y,t)=u^0(x)w_1(y,t)+v^0(x)w_2(y,t)$, where $w_i(\cdot, t)\in V_1$ solves
\[\langle\partial_{tt}w_i, \psi\rangle_{Y}+(\nabla_yw_i, \nabla_y\psi)_{\Sigma}=0\qquad\forall \psi\in V_1\]
with the initial conditions $w_i(\cdot, 0)=\delta_{i1}$ and $ \partial_t w_i(\cdot, 0)=\delta_{i2}$ with the Kronecker delta.
A standard energy argument shows that the solutions $w_1$ and $w_2$ are unique. As one easily sees, we have $w_1\equiv 1$ and $w_2(y,t)=t$ as solutions, which finishes the proof.
\end{proof}

\subsection{Changes in higher dimensions}\label{subsec:asymptotic:higherd}
In higher space dimensions, homogenization results are more involved.
In the low-contrast case, the formula for $\hat a$ involves the solution of additional PDEs on the unit cell $Y$. As in the stationary case one has \cite{BeLiPa1978}
\begin{equation}\label{eq:hommatrix}
	(\hat{a})_{ij}=\int_Y a(e_i+\nabla_y \xi_i)\cdot(e_j+\nabla_y \xi_j)\, dy, \qquad i,j=1,\ldots,d,
\end{equation}
where $e_i$ are the canonical basis vectors of $\mathbb{R}^d$ and $a$ is the $Y$-periodic function such that $a_\epsilon=a(x/\epsilon)$. Further, $\xi_i$, $i=1,\ldots, d$, is the unique $Y$-periodic $H^1$-function with zero mean, written as $\xi_i\in H^1_{\#, 0}(Y)$, that solves the so-called cell problem
\begin{equation}\label{eq:cellproblems}
	\int_Y a(e_i+\nabla_y \xi_i)\cdot \nabla_y v=0\qquad \text{for all}\quad v\in H^1_{\#, 0}(Y).
\end{equation}
Note that even though we assumed $a_\epsilon$ to be scalar-valued, the homogenized coefficient $\hat a$ is matrix-valued in general.
Moreover, even in our configuration \eqref{eq:aeps} for $a_\epsilon$, we can no longer give a closed formula for $\hat a$.

The results become even more intricate for the high contrast case with $a_0=\epsilon^2$ in \eqref{eq:wave_eq_variational_form}. According to \cite{NaSi2016},  $u_\epsilon$ two-scale converges to $\hat u(x, t)+\tilde u(x,y,t)$, where $\hat u\in L^\infty(0,T,H^1_0(\Omega))$ with $\partial_t \hat u\in L^\infty(0, T,L^2(\Omega))$, $\partial_{tt}\hat u\in L^\infty(0,T, H^{-1}(\Omega))$ and $\tilde u\in L^2((0,T)\times \Omega, H^1_0(\Sigma))$ with $\partial_t \tilde u\in L^2((0,T)\times \Omega\times \Sigma)$, $\partial_{tt}\tilde u\in L^2((0,T)\times \Omega, H^{-1}(\Sigma))$. The pair $(\hat u, \tilde u)$ is the unique solution to
\begin{equation}\label{eq:2schighcontrast}
\begin{array}{rllll}
\langle \partial_{tt}\hat u+\partial_{tt}\tilde u, \psi+\psi_1\rangle_{\Omega\times \Sigma}&+&(\hat a\nabla \hat u, \nabla \psi)_{L^2(\Omega)}+(a\nabla_y\tilde u, \nabla_y\psi_1)_{L^2(\Omega\times \Sigma)}&=&(f,\psi)_\Omega
\end{array}
\end{equation}
for all $(\psi,\psi_1)\in H^1_0(\Omega)\times L^2(\Omega, H^1_0(\Sigma))$ and almost all $t\in[0,T]$ with the initial conditions
\begin{equation}
\left\{
\begin{array}{rclrcl}
\hat u(\cdot, 0)&=&(1-|\Sigma|)u^0, &\qquad \tilde u(\cdot, \cdot, 0)&=&|\Sigma|u^0\\
\partial_t\hat u(\cdot, 0)&=&(1-|\Sigma|)v^0,&\qquad  \partial_t\tilde u(\cdot, \cdot, 0)&=&|\Sigma|v^0
\end{array}
\right.
\end{equation}
Here, $\hat a$ is defined similar to \eqref{eq:hommatrix}--\eqref{eq:cellproblems}, but the cell problems and the integration only occur in $Y\setminus \overline{\Sigma}$.
This is a coupled system for $\hat u$ and $\tilde u$ and, as \cite{NaSi2016} explains, it remains unclear whether or how it can be decoupled into a truly homogenized system for $\hat u$ alone with an (additional) homogenized coefficient derived from $\tilde u$.
In other words, when setting $\psi\equiv 0$ in \eqref{eq:2schighcontrast}, we cannot solve the resulting problem for $\tilde u$ depending only on $\hat u$, since $\partial_{tt}\hat u$ still appears. 
This is strikingly different from the stationary high-contrast case, where an additional homogenized coefficient arises, cf. \cite{All1992}.
Moreover, the system cannot be decoupled even in the case $f\equiv 0$. In particular, a more complex behavior of the homogenized solution $\hat u$ is expected in higher space dimensions than in the one-dimensional case discussed above.

\subsection{Numerical illustrations}\label{subsec:asymptotic:numexp}
In the following, we illustrate the differences between low- and high-contrast settings and one- versus two-dimensional settings by (direct) numerical simulations of \eqref{eq:wave_eq_variational_form}. We use linear finite elements on a mesh resolving the periodicity of $a_\epsilon$ and the implicit mid-point rule for time stepping.
Experiments were conducted in \texttt{python} using the module \texttt{scikit-fem} \cite{skfem2020} for the finite element stiffness and mass matrices.

First, we investigate the homogenization error in the one-dimensional setting depending on $\epsilon$ and the choice of $a_0$. For this, we numerically compute $u_\epsilon$ as solution to \eqref{eq:wave_eq_variational_form} and $\hat u$ as solution to \eqref{eq:wave_eq_hom} and consider their difference in the $L^\infty(0,T, L^2(\Omega))$-norm.
Let $\Omega = (0,1)$ and $T=0.25$. We use a uniform grid with resolution $h=2^{-13}$ and the time step size $\tau = 2^{-9}$. We consider $f=0$, $u^0(x)=\exp(-\frac{(x-0.5)^2}{\sigma^2})$ with $\sigma = 0.1$, and $v^0=0$. 
We choose an $\epsilon$-periodic coefficient taking the values 1 or $a_0$, i.e.,
\begin{equation*}
a_{\epsilon}(x) = a\bigl(\frac{x}{\epsilon}\bigr)\quad \text{ with }\quad a(y) = \left\{ \begin{array}{ll} a_0 & \text{if \ } \frac{1}{4} < y \leq \frac{3}{4} \\  1& \text{if \ } 0< y \leq \frac{1}{4} \text{ or }\frac34< y\leq 1  \end{array} \right.
\end{equation*}

\begin{figure}
	\begin{subfigure}{0.5\textwidth}
		\includegraphics[width=\textwidth]{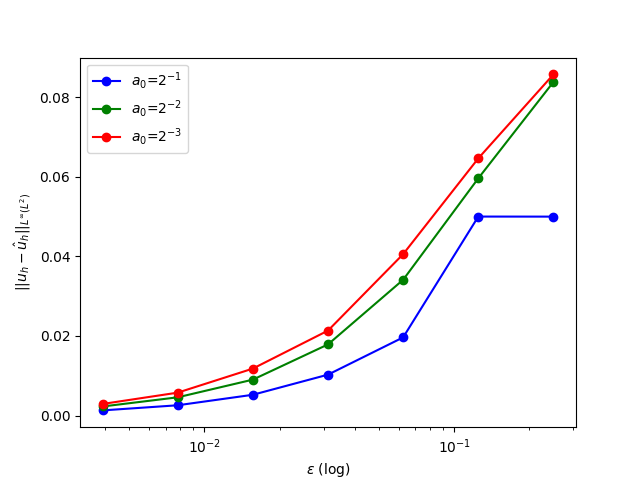}
		\captionsetup{justification=centering}
		\caption{Error vs.~ $\epsilon$}
		\label{fig:enmax_eps}
	\end{subfigure}
	\begin{subfigure}{0.5\textwidth}
		\includegraphics[width=\textwidth]{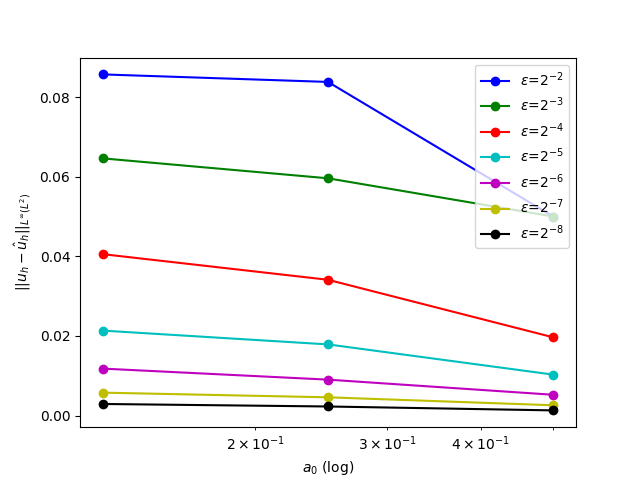}
		\captionsetup{justification=centering}
		\caption{Error vs.~$a_0$}
		\label{fig:enmax_a0}
	\end{subfigure}
	\caption{Homogenization error in $L^\infty(0,T, L^2(\Omega))$-norm for different values of $a_0$ and $\epsilon$}
	\label{fig:enmax}
\end{figure}

As expected, the homogenization error $u_\epsilon-\hat u$ decreases with decreasing $\epsilon$ until the temporal and spatial discretization errors dominate, see Figure \ref{fig:enmax_eps}. This behavior seems rather independent from the chosen value of $a_0$, but note that the depicted values still entail a somewhat moderate contrast in $a_\epsilon$. To investigate this further, we depict the homogenization error in dependence of $a_0$ for different $\epsilon$ in Figure \ref{fig:enmax_a0}. For fixed $\epsilon$, the considered error slightly increases when $a_0$ decreases. 
Figure \ref{fig:enmax} is in line with the formal observation above that the formula \eqref{eq:ahat1d} for $\hat a$ tends to $0$ for $a_0\to 0$ and thereby formally coincides with the asymptotic result in the high contrast setting from Proposition \ref{prop:homhighcontrast1d}.
We may conclude that in 1d, $\hat a$ from \eqref{eq:ahat1d} seems to give a good homogenization formula in the low- as well as the high-contrast setting.

\begin{figure}
	\centering
	\includegraphics[width=0.5\textwidth]{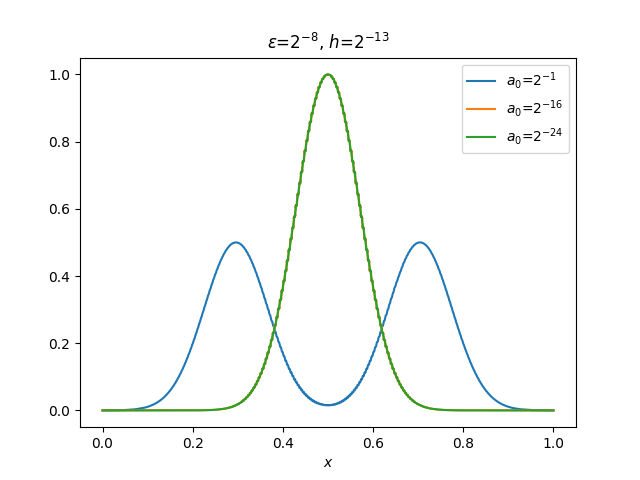} 
	\caption{Solutions for $\epsilon=2^{-8}$ and different values of $a_0$ at time $T=0.25$}
	\label{fig:sol_eps_varies_hc} 
\end{figure}

Second, we aim to illustrate the findings of Proposition \ref{prop:homhighcontrast1d} more. Precisely, we contrast the behavior of $u_\epsilon$ for different choices of $a_0$ in Figure \ref{fig:sol_eps_varies_hc} for $T=0.25$. We use the same setting as in the previous paragraph. 
The orange line, which lies on top of the green one, corresponds to the high-contrast case $a_0=\epsilon^{2}$ studied in Proposition \ref{prop:homhighcontrast1d}. In fact, $u_\epsilon$ does not seem to move, or in other words, it is very close to the initial value. This should be contrasted with the moving pulse (in blue) for the low contrast case.
Figure \ref{fig:sol_eps_varies_hc} suggests that the limit formula from Proposition \ref{prop:homhighcontrast1d} may also hold for other scalings of $a_0$, namely $a_0=\epsilon^p$ with $p\geq 2$, since the green curve corresponds to $p=3$.

\begin{figure}
	\begin{subfigure}{0.3\textwidth}
		\includegraphics[scale=0.4]{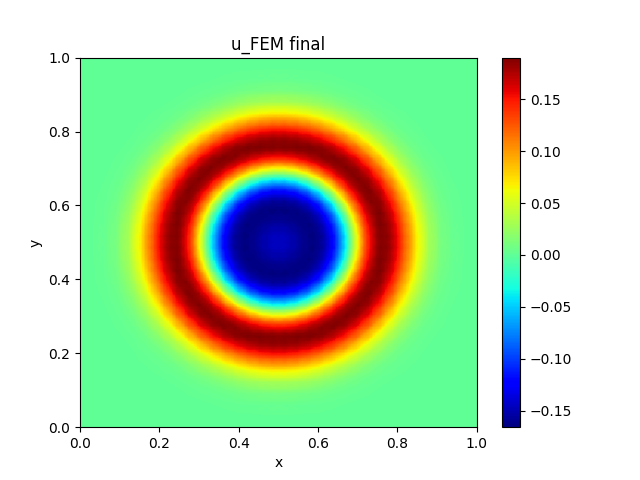}
		\captionsetup{justification=centering}
		\caption{$a_0=2^{-1}$}
		\label{fig:sol_2D_a0_2^-1_eps_2^-5}
	\end{subfigure}
	\begin{subfigure}{0.3\textwidth}
		\includegraphics[scale=0.4]{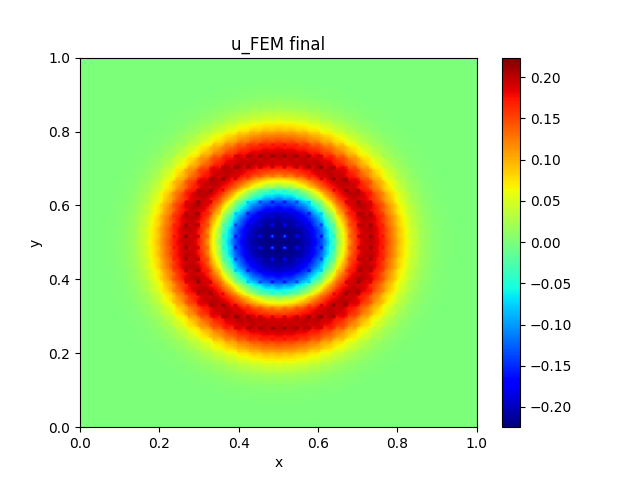}
		\captionsetup{justification=centering}
		\caption{$a_0=2^{-5}$}
		\label{fig:sol_2D_a0_2^-5_eps_2^-5}
	\end{subfigure}
	\begin{subfigure}{0.3\textwidth}
		\includegraphics[scale=0.4]{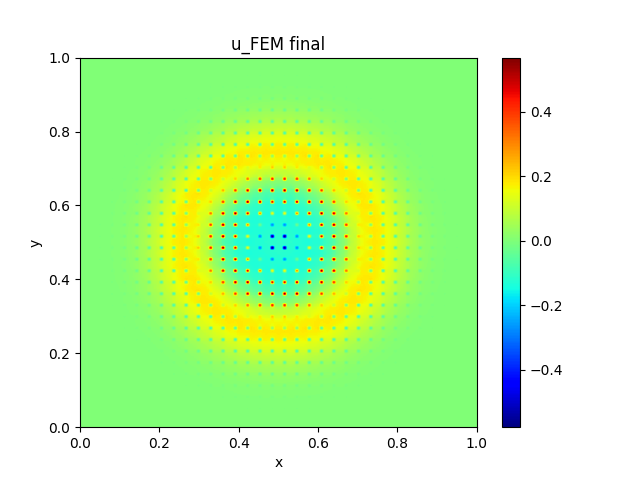}
		\captionsetup{justification=centering}
		\caption{$a_0=2^{-10}$}
		\label{fig:sol_2D_a0_2^-10_eps_2^-5}
	\end{subfigure}
	\caption{Solutions for different values of $a_0$ with $\epsilon = 2^{-5}$}
	\label{fig:sol_2D_eps_2^-5}
\end{figure}

Third, we aim to illustrate that the high-contrast behavior described in the previous paragraph and justified theoretically via Proposition \ref{prop:homhighcontrast1d} is specific to the one-dimensional setting. Therefore, we now consider the two-dimensional case, where Section \ref{subsec:asymptotic:higherd} predicts a more complex (limit) behavior. 
We choose $\Omega = (0,1)^2$ and $T=0.25$. We consider $h=2^{-7}$, $\tau = 2^{-9}$, $f=0$, $u^0(x,y)=\exp(-\frac{(x-0.5)^2}{\sigma^2} - \frac{(y-0.5)^2}{\sigma^2})$ with $\sigma = 0.1$, and $v^0=0$. 
The coefficient $a_\epsilon$ is set as follows
\begin{equation*}
a_{\epsilon}(x) = a\bigl(\frac{x}{\epsilon}\bigr)\quad \text{ with }\quad a(y) = \left\{ \begin{array}{ll} a_0 & \text{if \ } y\in \bigl(\frac14, \frac34\bigr)^2 \\  1& \text{if \ } y\in (0,1)^2\setminus \bigl(\frac14, \frac34\bigr)^2  \end{array} \right.
\end{equation*}
We call inclusions the parts of the domain where $a_\epsilon = a_0$. The amplitude of the solution is greater inside the inclusions (see Figure \ref{fig:sol_2D_eps_2^-5}). For fixed  $\epsilon$, the difference of amplitudes between the inclusions and the rest of the domain is more and more visible when $a_0$ decreases.  In contrast to the one-dimensional setting, the wave still propagates notably in the high-contrast case.

\section{Numerical approach beyond periodicity}\label{sec:numeric}
In this section, we drop the assumption on the periodic set-up of $\Omega_\epsilon$. One should still imagine that $\Omega_\epsilon$ consists of about $O(\epsilon^{-d})$ subdomains of diameter $O(\epsilon)$, but they are not necessarily arranged in a periodic fashion.
We present and analyze a numerical (multiscale) method that provides an approximation of $u_\epsilon$ on a coarse scale $H\gg \epsilon$ in Section \ref{subsec:numeric:lod}. This scale is typically introduced by the (coarse) mesh size of the triangulation.
Precisely, our approach relies on the Localized Orthogonal Decomposition (LOD), which was studied for the low-contrast case in \cite{AbHe2016}. In the stationary as well as the time-harmonic setting, the high-contrast case has been covered in \cite{PeSc2016,HeMa2017} and \cite{PeVe2020}, respectively. We will describe the required adaptions of the method, which introduce some implicit assumptions on the structure of $\Omega_\epsilon$ (at least for the theory).
The main goal is the numerical analysis of the LOD for the high-contrast wave equation in Section \ref{subsec:numeric:error}, which naturally combines techniques from \cite{AbHe2016} and \cite{PeSc2016,PeVe2020}. Numerical experiments in Section \ref{subsec:numeric:numexp} illustrate the results.

\subsection{Localized Orthogonal Decomposition}\label{subsec:numeric:lod}
The LOD is a spatial discretization method, which we combine with a standard time stepping scheme. For self-consistency of the material, we briefly review the LOD for the wave equation, closely following the presentation and notation of \cite{AbHe2016}.
\paragraph*{Spatial discretization.} We introduce a (coarse) rectangular\footnote{Since our implementation is based on rectangular meshes, we restrict the presentation to this case. However, results transfer to general parallelepiped meshes in higher dimensions as well as to simplicial meshes.} mesh $\mathcal{T}_H$ with mesh size $H$ and a fine rectangular mesh  $\mathcal{T}_h$, where the latter resolves the variations of $a_\epsilon$. All meshes are assumed to be shape-regular (in the sense that the smallest angle is bounded from below), $\mathcal{T}_H$ is also quasi-uniform, and $\mathcal{T}_h$ is a refinement of $\mathcal{T}_H$.
Let $V_H = Q_1(\mathcal{T}_H) \cap H_0^1(\Omega)$ and $V_h = Q_1(\mathcal{T}_h) \cap H_0^1(\Omega)$ be the associated lowest-order finite element spaces, where $Q_1(\mathcal{T}_h)$ is the space of possibly discontinuous functions that are polynomials of coordinate degree at most one on each element of $\mathcal{T}_h$.
Let $I_H: H^1_0(\Omega)\to V_H$ be a stable interpolation operator. Necessary assumptions on $I_H$ as well as possible choices are discussed further below. Let $W_h=\operatorname{ker}(I_H)$.

The corrector $\mathcal{C}_{h,\Omega} : V_H \rightarrow W_h$ is the solution $\mathcal{C}_{h,\Omega}(v_H) \in W_h$ of 
\[(a_\epsilon \nabla \mathcal{C}_{h,\Omega}(v_H),\nabla w_h)_{L^2(\Omega)} = - (a_\epsilon \nabla v_H,\nabla w_h)_{L^2(\Omega)}\quad \text{for all} \quad w_h \in W_h.\]
For a coarse element $K \in \mathcal{T}_H$, a patch $U_m(K)$ is defined as follows
\begin{equation}
\begin{array}{rll}
U_0(K) &=&  K,\\
U_m(K) &=& \bigcup \{ T \in \mathcal{T}_H | T \bigcap U_{m-1}(K) \neq \O \emptyset \}, \quad m \geq 1.
\end{array}
\end{equation}
We define the space $W_h(U_m(K)) = \{ w_h \in W_h | w_h = 0 \text{ in } \Omega \backslash U_m(K) \}$ and the localized correctors $\mathcal{C}_{h,m}^K$ as follows : for $v_H \in V_H$, $\mathcal{C}_{h,m}^K(v_H) \in W_h(U_m(K))$ solves
\begin{equation}
\int_{U_m(K)} a_\epsilon \nabla \mathcal{C}_{h,m}^K(v_H) \cdot \nabla w_h = - \int_K a_\epsilon \nabla v_H \cdot \nabla w_h \text{\ \ \ } \forall w_h \in W_h(U_k(K))
\end{equation}
Let the global approximation $\mathcal{C}_{h,m}$ of $\mathcal{C}_{h, \Omega}$ be defined by $\mathcal{C}_{h,m}(v_H) = \sum_{K\in \mathcal{T}_H} \mathcal{C}_{h,m}^K(v_H)$. 
The LOD space is now $V_{H,m}^{ms} = \{v_H + \mathcal{C}_{h,m}(v_H) | v_H \in V_H \}$.

We introduce the following notations
\begin{equation*}
\begin{array}{ccc}
b_{H,m}(v_H,w_H) & = & b_\epsilon(v_H + \mathcal{C}_{h,m}(v_H),w_H + \mathcal{C}_{h,m}(w_H)), \\
(v_H,w_H)_{H,m} & = & (v_H + \mathcal{C}_{h,m}(v_H),w_H + \mathcal{C}_{h,m}(w_H))_{L^2(\Omega)}
\end{array}
\end{equation*}
Let $u_{H,m} \in H^2(0, T, V_H)$ denote the semi-discrete approximation. It solves the following system for almost all $t>0$ and $v_H \in V_H$

\begin{equation}
\label{eq:wave_eq_semi_discrete_LOD}
\left\{
\begin{array}{rll}
( \partial_{tt} u_{H,m}(.,t),v_H )_{H,m} + b_{H,m}(u_{H,m}(.,t), v_H) &=& (f(.,t),v_H + \mathcal{C}_{h,m}(v_H))_{L^2(\Omega)} \\ 
(u_{H,m} + \mathcal{C}_{h,m}(u_{H,m}))(.,0) &=& \pi_{H,m}^{ms}(u^0)\\
\partial_t(u_{H,m} + \mathcal{C}_{h,m}(u_{H,m}))(.,0) &=& v^0_{H,m}\\
\end{array}
\right.
\end{equation}
where $\pi_{H,m}^{ms}: H_0^1(\Omega) \rightarrow V_{H,m}^{ms}$ is the elliptic projection on $V_{H,m}^{ms}$. The second initial value $v^0_{H,m}$ can be chosen as either $\pi_{H,m}^{ms}(v^0)$ or as $P_{H,m}^{ms}(v^0)$, where $P_{H,m}^{ms} : H_0^1(\Omega) \rightarrow V_{H,m}^{ms}$ is the $L^2$-projection on $V_{H,m}^{ms}$. Besides $u_{H,m}$ we also define the semi-discrete multiscale approximation $u_{H,m}^{ms}\in H^2(0,T, V_{H,m}^{ms})$ via $u_{H,m}^{ms}=(1+\mathcal C_{h,m})u_{H,m}$.

Let $\mathcal{N}_H$ denote the set of the interior vertices of the coarse grid.  $\phi_z \in V_H$ denotes the nodal basis function corresponding to the node $z \in \mathcal{N}_H$. $N$ refers to the number of interior coarse nodes.
We introduce the corrected stiffness matrix $S_m$, corrected mass matrix $M_m$ and corrected source vector $F_m$, respectively, as
\begin{equation*}
\begin{array}{lll}
(S_m)_{i,j} = b_{H,m}(\phi_{z_j}, \phi_{z_i}), & (M_m)_{i,j} = (\phi_{z_j}, \phi_{z_i})_{H,m}, & (F_m)_{i}(t) = (f(.,t), \phi_{z_i} + \mathcal{C}_{h,m}(\phi_{z_i})_{L^2(\Omega)}.
\end{array}
\end{equation*}
The initial vectors $\overline{u^0_m}$ and $\overline{v^0_m}$ are given by
\begin{equation}
(\overline{u^0_m})_i \text{ is such that } \pi_{H,m}^{ms}(u^0) = \sum_{i=1}^N (\overline{u^0_m})_i (\phi_{z_i} + \mathcal{C}_{h,m}(\phi_{z_i}))
\end{equation}    
\begin{equation}
(\overline{v^0_m})_i \text{ is such that } v^0_{H,m} = \sum_{i=1}^N (\overline{v^0_m})_i (\phi_{z_i} + \mathcal{C}_{h,m}(\phi_{z_i}))
\end{equation}
Note that $M_m$ and $S_m$ are time-independent and need to be assembled only once.
Expanding $u_{H,m}$ in the basis $(\phi_{z_i})_{1 \leq i \leq N}$, i.e., $u_{H,m}(.,t) = \sum_{i=1}^N (\zeta_m(t))_i \phi_{z_i}$, \eqref{eq:wave_eq_semi_discrete_LOD} can be written as a system of ODEs for $\zeta_m(t)$.

\paragraph*{Time discretization.} We can equally use the implicit midpoint rule or the Crank-Nicolson scheme from \cite{AbHe2016}. In the considered setting, both schemes differ only in the treatment of the right-hand side $f$ and are even equivalent if the time dependency of $f$ is affine.
Let $\tau$ be the time step  size and set $t^n=\tau n$ for $n\in \{0,\ldots, J=T/\tau\}$ and, similarly, $t^{n+\frac 12}=\tau (n+\frac 12)$. 
We are interested in finding an approximation $\zeta_m^n$ for $\zeta_m(t^n)$. To re-write the ODE system of second order for $\zeta_m$ in a system of first order, we introduce the auxiliary quantity $\eta_m=\Dot{\zeta}_m$. The time-discrete approximations $(\zeta_m^n, \eta_m^{n})$ to $(\zeta_m(t^n), \eta_m(t^{n}))$ are determined via the following system
\begin{equation}
\label{eq:wave_eq_space_time_discrete_LOD}
\begin{array}{rll}
\zeta_m^0 &=& \overline{u^0_m} \\
\eta_m^0 &=& \overline{v^0_m} \\
\left( M_m + \frac{\tau^2}{4} S_m \right) \eta_m^{n+\frac{1}{2}} &=& M_m \eta_m^n + \frac{\tau}{2} \left( - S_m \zeta_m^{n} + M_m F_m^{n+\frac{1}{2}} \right)\\
\zeta_m^{n+1} &=& \zeta_m^{n} + \tau \eta_m^{n+\frac{1}{2}} \\
\eta_m^{n+1}&=&2\eta_m^{n+\frac12}-\eta_m^n.
\end{array}
\end{equation}
Here, $F_m^{n+\frac 12}=F_m(t^{n+\frac 12})$ for the implicit midpoint rule and $F_m^{n+\frac 12}=\frac12(F_m(t^n)+F_m(t^{n+1}))$ for the Crank-Nicolson scheme.

Note that only the third equation requires the solution of a linear system to determine $\eta_m^{n+\frac 12}$. The involved matrix $M_m+\frac{\tau^2}{4}S_m$ is time-independent and  rather low-dimensional (corresponding to the number of degrees of freedom of $V_H$). Therefore, the LU decomposition of this matrix can be computed once and stored, such that the whole time stepping procedure is very cheap.

\paragraph{Interpolation operator.} As seen above, the interpolation operator in the LOD method is not merely a theoretical tool, but an integral part of the construction. In general, we assume that $I_H:H^1_0(\Omega)\to V_H$ is a projection, i.e., $I_H\circ I_H=I_H$. (Deviations are possible, but it facilitates the analysis). Furthermore, stability and approximation properties are required, which can be summarized in the low-contrast case as
\begin{equation*}
\|v-I_Hv\|_{L^2(K)}+H\|\nabla I_H v\|_{L^2(K)}\leq C_IH\|\nabla v\|_{L^2(U_1(K))}
\end{equation*}
for all $v\in H^1_0(\Omega)$ and all $K\in \mathcal{T}_H$.
Several choices for $I_H$ are possible and we refer to, e.g., \cite{MaPe2021} for a detailed discussion. For instance, \cite{AbHe2016} uses the $L^2$ projection.
Another admissible and popular choice that we also used in our simulations is $I_H=E_H\circ \Pi_H$, where $\Pi_H$ is the element-wise $L^2$-projection onto $Q_1(\mathcal{T}_H)$. Moreover, $E_H$ denotes an averaging operator, characterized for any $v_H\in Q_1(\mathcal{T}_H)$ and any vertex $z$ of $\mathcal{T}_H$ by
\begin{equation*}
(E_H(v_H))(z)=\sum_{K\in \mathcal{T}_H:z\in K}(v_H|_K)(z)\cdot \frac{1}{\operatorname{card}\{K^\prime \in \mathcal{T}_H: z\in K^\prime\}}.
\end{equation*}

In the high-contrast case, we require the above stability and approximation properties in norms weighted with $a_\epsilon$. Precisely, for the analysis below, the interpolation operator $I_H$ has to satisfy
\begin{equation}\label{eq:intpol}
\|\sqrt{a_\epsilon}\nabla I_H v\|_{L^2(K)}\leq C_I\|\sqrt{a_\epsilon}\nabla v\|_{L^2(U_1(K))}\quad \text{and}\quad \|\sqrt{a_\epsilon}w\|_{L^2(K)}\leq C_I H\|\sqrt{a_\epsilon}\nabla w\|_{L^2(U_1(K))}
\end{equation}
for all $v\in H^1_0(\Omega)$, all $w\in W_h$ and all $K\in \mathcal{T}_H$.
These conditions have been verified under certain geometric conditions on $\Omega_\epsilon$ using special $a_\epsilon$-weighted interpolation operators. 
For instance, \cite{PeSc2016} uses $a_\epsilon$ weighted $L^2$-projections onto $V_H$ on patches $\omega_z=\{K\in \mathcal{T}_H: z\in K\}$ and evaluates these interpolations at all vertices $z$ to define $(I_H(v))(z)$. Assumption \eqref{eq:intpol} is shown to hold if $a_\epsilon$ is quasi-monotone. This is for instance satisfied in the two-dimensional periodic case described in Section \ref{sec:asympotic} above. For details, we refer to  \cite{PeSc2016}. 
The geometric conditions for the Scott-Zhang-type interpolation in \cite{HeMa2017} are not so easy to describe, but examples mentioned therein include the periodic set-up of $\Omega_\epsilon$.
In our numerical experiments, we use the following interpolation operator.
We set $I_H=E_H\circ\Pi_{H, a}$, where $\Pi_{H,a}$ is the $a_\epsilon$ weighted element-wise $L^2$-projection onto $Q_1(\mathcal{T}_H)$. Precisely, for any $v\in H^1(K)$, $\Pi_{H, a}(v)\in Q_1(K)$ is given by
\begin{equation*}
\int_Ka_\epsilon\Pi_{H,a}(v)\psi_H\, dx=\int_K a_\epsilon v\psi_H\, dx\quad \text{for all}\quad \psi_H\in Q_1(K).
\end{equation*}
We note that, for this choice, \eqref{eq:intpol} has not been proved to the best of our knowledge, but the results of \cite{PeSc2016} and of \cite{BrXu1991} concerning weighted projections indicate a validity of \eqref{eq:intpol} in this case.
In the following analysis we do not require the explicit form of $I_H$, but only use  \eqref{eq:intpol}.

\subsection{Error analysis for high-contrast media}\label{subsec:numeric:error}
We are interested in transferring the a priori error results in \cite{AbHe2016} to the present high-contrast setting. In order to have $\epsilon$-explicit estimates, we have to carefully take into account any dependence on the contrast (i.e., the lower bound of $a_\epsilon$).
We write $\lesssim$ to denote inequality up to a multiplicative constant $C$, which does not depend on $\epsilon$ or the discretization parameters $H,h$ and $m$. Note that the constants may depend on the final time $T$. 

We will analyze errors for the semi-discrete method \eqref{eq:wave_eq_semi_discrete_LOD}. Estimates for the full discretization can then be deduced following \cite{AbHe2016}.
For the error analysis we require the following two assumptions.
\begin{Assumption}
	(a) The geometry of $\Omega_\epsilon$ is such that an interpolation operator $I_H$ satisfying \eqref{eq:intpol} exists. Further, we assume that this $I_H$ is used in the construction of the LOD.\\
	(b) The geometry of $\Omega_\epsilon$ is such that the following weighted Poincar\'e inequality holds: For all $v\in H^1_0(\Omega)$, there exists a constant $C_P$ independent of $\epsilon$ such that
	\begin{equation}\label{eq:weightedpoincare}
	\|v\|_{L^2(\Omega)}\leq C_P\|\sqrt{a_\epsilon}\nabla v\|_{L^2(\Omega)}.
	\end{equation}
\end{Assumption}
In the previous section, we already discussed the first part of the assumption. The weighted Poincar\'e inequality in the second part holds in the periodic case with compact inclusions, cf.~\cite[Rem.~4.2]{All1992}.

\paragraph{Ritz projection errors.} As usual, we start with estimating the error of the elliptic (Ritz) projection $\pi_{H,m}^{ms}$ onto $V_{H,m}^{ms}$. In \cite{HeMa2017,PeSc2016}, the following \emph{contrast-independent} estimate for the truncation of the correctors is shown.

\begin{Lemma}
	There exists a generic constant $0<\gamma<1$ (independent of $H$, $h$ and $\epsilon$) such that for any $v_H\in V_H$
	\begin{equation}\label{eq:truncerror}
	\|\sqrt{a_\epsilon}\nabla (\mathcal C_{h,m}-\mathcal C_{h,\Omega})(v_H)\|_{L^2(\Omega)}\lesssim m^{d/2}\gamma^m \|\sqrt{a_\epsilon}\nabla v_H\|_{L^2(\Omega)} 
	\end{equation}
\end{Lemma}
In the following, $\pi_h$ denotes the (elliptic) Ritz projection onto the fine finite FE space $V_h$. The next lemma gives the  error estimates for $\pi_{H,m}^{ms}$ in various norms.

\begin{Lemma}\label{lem:ritz}
	The following estimates hold for almost every $t\in [0,T]$.\\
	(a) For any $v\in L^1(H^1)$, we have
	\begin{equation}\label{eq:ritzweightedl2}
	\|\sqrt{a_\epsilon}(v(\cdot, t)-\pi_{H,m}^{ms}v(\cdot, t))\|_{L^2(\Omega)}\lesssim \|\sqrt{a_\epsilon}(v(\cdot, t)-\pi_hv(\cdot, t))\|_{L^2(\Omega)}+(H+m^{d/2}\gamma^m)\|\sqrt{a_\epsilon}\nabla v(\cdot, t)\|_{L^2(\Omega)}.
	\end{equation}
	(b) For the solution $u_\epsilon$ to \eqref{eq:wave_eq_variational_form} assume that $\partial_t^iu_\epsilon\in L^1(H^1_0)$ and $\partial_t^{2+i}u_\epsilon, \partial_t^if\in L^1(L^2)$ for $i\in \{0,1,2\}$. Then,
	\begin{align}\label{eq:ritzenergy}
	&\|\sqrt{a_\epsilon}\nabla(\partial_t^iu_\epsilon(\cdot, t)-\pi_{H,m}^{ms}(\partial_t^iu_\epsilon(\cdot, t)))\|_{L^2(\Omega)}\\\nonumber
	&\quad\lesssim \|\sqrt{a_\epsilon}\nabla(\partial_t^iu_\epsilon(\cdot, t)-\pi_h(\partial_t^iu_\epsilon(\cdot, t)))\|_{L^2(\Omega)}\\\nonumber
	&\qquad+(H+m^{d/2}\gamma^m)\bigl(\|(a_\epsilon)^{-1/2}(\partial_t^{i+2}u_\epsilon-\partial_t^i f)\|_{L^2(\Omega)}+\|\sqrt{a_\epsilon}\nabla\partial_t^iu_\epsilon\|_{L^2(\Omega)}\bigr),\\\label{eq:ritzl2}
	&\|\partial_t^iu_\epsilon(\cdot, t)-\pi_{H,m}^{ms}(\partial_t^iu_\epsilon(\cdot, t))\|_{L^2(\Omega)}\\\nonumber
	&\quad\lesssim \|\partial_t^iu_\epsilon(\cdot, t)-\pi_h(\partial_t^iu_\epsilon(\cdot, t))\|_{L^2(\Omega)}\\\nonumber
	&\qquad +(H+m^{d/2}\gamma^m)\bigl(\|(a_\epsilon)^{-1/2}(\partial_t^{i+2}u_\epsilon-\partial_t^i f)\|_{L^2(\Omega)}+\|\sqrt{a_\epsilon}\nabla\partial_t^iu_\epsilon\|_{L^2(\Omega)}\bigr).
	\end{align}
\end{Lemma}

\begin{proof}
For simiplicity, we will omit time dependencies throughout the proof if no confusion can arise.\\
\emph{Proof of \eqref{eq:ritzweightedl2}:}
Let $e_{H,m}=\pi_{H,m}^{ms}v-\pi_h v$. We introduce the dual solution $z\in H^1_0(\Omega)$ such that
\[b_\epsilon(\psi, z)=(e_{H,m}, \psi)_a.\]
Note that it holds
\begin{equation*}
\|\sqrt{a_\epsilon}\nabla z\|_{L^2(\Omega)}\lesssim \|\sqrt{a_\epsilon}e_{H,m}\|_{L^2(\Omega)}
\end{equation*}
by standard elliptic a priori (stability) estimates using \eqref{eq:weightedpoincare}.
Then we obtain
\begin{align*}
\|e_{H,m}\|_{L^2_a(\Omega)}^2&=b_\epsilon(e_{H,m}, z)=b_\epsilon(e_{H,m}, z-z_{H,m})\\
&\lesssim \|\sqrt{a_\epsilon}\nabla e_{H,m}\|_{L^2(\Omega)}\|\sqrt{a_\epsilon}\nabla(z-z_{H,m})\|_{L^2(\Omega)}
\end{align*}
for any $z_{H,m}\in V_{H,m}^{ms}$.
We choose $z_{H,m}=(1+\mathcal C_{h,m})I_H z=(1+\mathcal C_{h,\Omega})z+(\mathcal C_{h,m}-\mathcal C_{h,\Omega})I_H z$.
Because of \eqref{eq:truncerror} and \eqref{eq:intpol}, we deduce for the second term
\begin{align*}
\|\sqrt{a_\epsilon}\nabla (\mathcal C_{h,m}-\mathcal C_{h,\Omega})I_H z\|_{L^2(\Omega)}\lesssim m^{d/2}\gamma^m\|\sqrt{a_\epsilon}\nabla z\|_{L^2(\Omega)}\lesssim m^{d/2}\gamma^m\|\sqrt{a_\epsilon}e_{H,m}\|_{L^2(\Omega)},
\end{align*}
where we used the a priori (stability) estimate for $z$ in the last step. 
For the first term $z-(1+\mathcal C_{h,\Omega})z=-\mathcal C_{h,\Omega} z$ we obtain
\begin{align*}
\|\sqrt{a_\epsilon}\nabla \mathcal C_{h,\Omega} z\|_{L^2(\Omega)}^2&\leq b_\epsilon(\mathcal C_{h,\Omega}z, \mathcal C_{h,\Omega} z)=b_\epsilon(\mathcal C_{h,\Omega} z, z)=(e_{H,m}, \mathcal C_{h,\Omega} z)_a\\
&\leq \|\sqrt{a_\epsilon}e_{H,m}\|_{L^2(\Omega)}\|\sqrt{a_\epsilon}\mathcal C_{h,\Omega}z\|_{L^2(\Omega)}\\
&\lesssim H\|\sqrt{a_\epsilon}e_{H,m}\|_{L^2(\Omega)}\|\sqrt{a_\epsilon}\nabla \mathcal C_{h,\Omega}z\|_{L^2(\Omega)},
\end{align*}
where we used $\mathcal C_{h,\Omega} z\in W_h$ and \eqref{eq:intpol} in the last step.
Collecting the results, we have shown that
\[\|\sqrt{a_\epsilon}e_{H,m}\|_{L^2(\Omega)}\lesssim (H+m^{d/2}\gamma^m)\|\sqrt{a_\epsilon}\nabla e_{H,m}\|.\]
The stability of the Ritz projections $\pi_{H,m}^{ms}$ and $\pi_h$ as well as the triangle inequality finish the proof of \eqref{eq:ritzweightedl2}.

\emph{Proof of \eqref{eq:ritzenergy}:} Introduce the simplified notation $v_\epsilon=\partial_t^i u_\epsilon$ and $\overline{f}=\partial_t^i f$. Observe that $v_\epsilon$ solves
\[(\partial_{tt}v_\epsilon(\cdot, t), \psi)+b_\epsilon(v_\epsilon(\cdot, t), \psi)=(\overline f(\cdot, t), \psi)\qquad \forall \psi\in H^1_0(\Omega).\]
By the definition of the Ritz projections, we see that
\[\|\sqrt{a_\epsilon}\nabla(\pi_hv_\epsilon-\pi_{H,m}^{ms}v_\epsilon)\|_{L^2(\Omega)}\leq \inf_{ v_{H,m}\in V_{H,m}^{ms}}\|\sqrt{a_\epsilon}\nabla(\pi_hv -v_{H,m}^{ms})\|_{L^2(\Omega)}.\]
As in the previous step, we choose $v_{H,m}=(1+\mathcal C_{h,m})I_H\pi_hv_\epsilon$ and estimate
\begin{align*}
\|\sqrt{a_\epsilon}\nabla(\pi_{H,m}^{ms}v_\epsilon-\pi_hv_\epsilon)\|_{L^2(\Omega)}&\leq \|\sqrt{a_\epsilon}\nabla\mathcal C_{h,\Omega}\pi_hv_\epsilon\|_{L^2(\Omega)}+\|\sqrt{a_\epsilon}\nabla(\mathcal C_{h,m}-\mathcal C_{h,\Omega})\pi_hv_\epsilon\|_{L^2(\Omega)}\\
&\lesssim \|\sqrt{a_\epsilon}\nabla\mathcal C_{h,\Omega}\pi_hv_\epsilon\|_{L^2(\Omega)} +m^{d/2}\gamma^m \|\sqrt{a_\epsilon}\nabla v_\epsilon\|_{L^2(\Omega)},
\end{align*}
where we used \eqref{eq:truncerror} and the stability of $\pi_h$ in the last step. For $\mathcal C_{h,\Omega}\pi_h v_\epsilon$ we obtain with the definition of $\mathcal C_{h,\Omega}$ and $\pi_h$ that
\begin{align*}
\|\sqrt{a_\epsilon}\nabla \mathcal C_{h,\Omega}\pi_hv_\epsilon\|_{L^2(\Omega)}^2&\leq b_\epsilon(\mathcal C_{h,\Omega}\pi_hv_\epsilon, \mathcal C_{h,\Omega}\pi_hv_\epsilon)=b_\epsilon(v_\epsilon, \mathcal C_{h,\Omega}\pi_hv_\epsilon)\\
&=(\overline{f}-\partial_{tt}v_\epsilon, \mathcal C_{h,\Omega}\pi_hv_\epsilon-I_H \mathcal C_{h,\Omega}\pi_hv_\epsilon)_{L^2(\Omega)}\\
&\lesssim H\|(a_\epsilon)^{-1/2}(\overline f-\partial_{tt}v_\epsilon)\|_{L^2(\Omega)}\|\sqrt{a_\epsilon}\nabla \mathcal C_{h,\Omega}\pi_hv_\epsilon\|_{L^2(\Omega)},
\end{align*}
where we employed \eqref{eq:intpol} in the last step. Combination of the foregoing estimates and the triangle inequality yield the result.
 
\emph{Proof of \eqref{eq:ritzl2}:} Follows from the weighted Poincar\'e inequality \eqref{eq:weightedpoincare} and \eqref{eq:ritzenergy}.
\end{proof}

Note that we have the term  $\|(a_\epsilon)^{-1/2}(\overline f-\partial_{tt} v_\epsilon)\|_{L^2(\Omega)}$ on the right-hand side of \eqref{eq:ritzenergy}--\eqref{eq:ritzl2} because the application of \eqref{eq:intpol} to $\mathcal C_{h, \Omega}\pi_hv_\epsilon$ requires the weighted $L^2$-norm. This agrees with the error estimates for stationary source problems, see, e.g., \cite{PeVe2020}.
We discuss bounds on $\|(a_\epsilon)^{-1/2}(\overline f-\partial_{tt} v_\epsilon)\|_{L^2(\Omega)}$ in Remark \ref{rem:timeregest}.
By combining the proof for \eqref{eq:ritzweightedl2} with \eqref{eq:ritzenergy}, we obtain a quadratic rate for the Ritz projection error of $v_\epsilon$ in the weighted $L^2$-norm. This bound is, however, not needed in our main arguments below.

\paragraph*{Semi-discrete error analysis.} We employ the usual splitting of $u-u_{H,m}^{ms}$ into the Ritz projection error $\rho=u-\pi_{H,m}^{ms}u$, for which we can apply Lemma \ref{lem:ritz}, and the remainder term $\eta=\pi_{H,m}^{ms}u-u_{H,m}^{ms}$, which belongs to the finite-dimensional space $V_{H,m}^{ms}$.

\begin{Theorem}\label{thm:error}
	(a) Assume that $\partial_t^3u_\epsilon,\partial_tf\in L^\infty(L^2)$, $\partial_t^4u_\epsilon,\partial_{tt}f\in L^1(L^2)$, $\partial_tu_\epsilon\in L^\infty(H^1_0)$ and $\partial_{tt}u_\epsilon\in L^1(H^1_0)$. If we choose $v^0_{H,m}=\pi_{H,m}^{ms}(v^0)$ in \eqref{eq:wave_eq_semi_discrete_LOD}, then
	\begin{equation}\label{eq:errorenergy}
	\begin{aligned}
	&\!\!\!\!\|\partial_t u_\epsilon -\partial_t u_{H,m}^{ms}\|_{L^\infty(L^2)}+\|\sqrt{a_\epsilon}\nabla (u_\epsilon-u_{H,m}^{ms})\|_{L^\infty(L^2)}\\
	&\lesssim (H+m^{d/2}\gamma^m)\Bigl(\sum_{i=0}^1\|(a_\epsilon)^{-1/2}(\partial_t^{i+2}u_\epsilon-\partial_t^if)\|_{L^\infty(L^2)}+\|(a_\epsilon)^{-1/2}(\partial_t^4u_\epsilon-\partial_{tt}f)\|_{L^1(L^2)}\\
	&\qquad+\sum_{i=0}^1\|\sqrt{a_\epsilon}\nabla \partial_t u_\epsilon\|_{L^\infty(L^2)}+||\sqrt{a_\epsilon}\nabla \partial_{tt}u_\epsilon\|_{L^1(L^2)}\Bigr) +e^1_{\mathrm{disc}}(h)
	\end{aligned}
	\end{equation}
	with the fine-scale discretization error \[e^1_{\mathrm{disc}}(h)=\|(1-\pi_h)\partial_tu_\epsilon\|_{L^\infty(L^2)}+\|\sqrt{a_\epsilon}\nabla(1-\pi_h)u_\epsilon\|_{L^\infty(L^2)}+\|(1-\pi_h)\partial_{tt}u_\epsilon\|_{L^\infty(L^2)}.\]
	(b) Assume that $\partial_{tt}u_\epsilon\in L^\infty(L^2)$, $\partial_t^3u_\epsilon, \partial_t f\in L^1(L^2)$, and $\partial_tu_\epsilon\in L^1(H^1_0)$, then
	\begin{equation}\label{eq:errorl2}
	\begin{aligned}
	&\|u_\epsilon - u_{H,m}^{ms}\|_{L^\infty(L^2)}\\
	&\quad\lesssim (H+m^{d/2}\gamma^m)\bigl(\|(a_\epsilon)^{-1/2}(\partial_{tt}u_\epsilon-f)\|_{L^\infty(L^2)}+\|(a_\epsilon)^{-1/2}(\partial_{t}^3u_\epsilon-\partial_t f)\|_{L^1(L^2)}\\
	&\qquad+\|\sqrt{a_\epsilon}\nabla u_\epsilon\|_{L^\infty(L^2)}+\|\sqrt{a_\epsilon}\nabla \partial_t u_\epsilon\|_{L^1(L^2)}\bigr)+e^2_{\mathrm{disc}}(h),
	\end{aligned}
\end{equation}
with the fine-scale discretization error $e^2_{\mathrm{disc}}(h)=\|(1-\pi_h)u_\epsilon\|_{L^\infty(L^2)}+\|(1-\pi_h)\partial_tu_\epsilon\|_{L^1(L^2)}$.
\end{Theorem}

\begin{Remark}[Estimating the right-hand side of \eqref{eq:errorenergy}--\eqref{eq:errorl2}]\label{rem:timeregest}
	With respect to the $\varepsilon$-dependency, the most critical terms in \eqref{eq:errorenergy}--\eqref{eq:errorl2} are the first two terms each as they are weighted with $(a_\epsilon)^{-1/2}\sim \epsilon^{-1}$.
	From the consideration of source problems, cf.~\cite{PeVe2020}, we expect the term $(a_\epsilon)^{-1/2}f$. Note that the critical $\epsilon$-dependency vanishes if $f$ has support outside $\Omega_\epsilon$. For the Helmholtz problem, \cite{PeVe2020} further discusses that the occurrence $(a_\epsilon)^{-1/2}f$ in the error estimates can be interpreted in the light of homogenization theory.
	To estimate the terms $(a_\epsilon)^{-1/2}\partial_t^iu_\epsilon$ in \eqref{eq:errorenergy}--\eqref{eq:errorl2}, we bound $(a_\epsilon)^{-1/2}$ by $\epsilon^{-1}$ and use (time-)regularity estimates following from \cite[Prop.~4.4]{AbHe2016} and \cite[Ch.~7.2]{Evans}.
	Compactly written, one obtains for $i\in \mathbb{N}_0$ that
	\[\|(a_\epsilon)^{-1/2}\partial_t^{i+1}u_\epsilon\|_{L^\infty(L^2)}\lesssim_T\epsilon^{-1}\bigl(\|\partial_t^if\|_{L^2(L^2)}+\|w_i\|_{H^1(\Omega)}+\|w_{i+1}\|_{L^2(\Omega)}),\]
	where $w_i$ is inductively defined via $w_0=u^0$, $w_1=v^0$ and $w_i=\partial_t^if(\cdot, 0)+\nabla\cdot(a_\epsilon\nabla w_{i-1}(\cdot, 0))$ for $i\geq 2$.
	In particular, for zero initial data and if time derivatives of $f$ are zero at $t=0$, the terms $w_i$ vanish and the time regularity estimate is left with $\epsilon^{-1}\|\partial_t^if\|_{L^2(L^2)}$ on the right-hand side. Such an $\epsilon$-dependency in front of $f$ also shows up from our previous discussion of the source problem if $f$ has some support in $\Omega_\epsilon$.
	
	Note that similar time regularity estimates also allow to bound the third and fourth term in \eqref{eq:errorenergy}--\eqref{eq:errorl2} each. Since the $\epsilon$-dependency for these terms is not critical, we, however, do not discuss them in further detail. 
	
	Summarizing, we can further estimate \eqref{eq:errorenergy}--\eqref{eq:errorl2} using time regularity estimates similar as in \cite{AbHe2016}. In particular, the required time regularity of $u_\epsilon$ can be achieved for so-called compatible initial data, i.e., if $w_i$ up to the required index $i$ exist and lie in the necessary spaces $H^1_0(\Omega)$ or $L^2(\Omega)$, cf.~\cite{AbHe2016} for precise statements. 
	However, such compatible initial data may still induce (critical) $\epsilon$-dependencies on the right-hand side in our high-contrast case. Completely $\epsilon$-independent bounds can be achieved if the initial data as well as $\partial_t^if(\cdot, 0)$ are zero and if $f$ (and its time derivatives) have a small (i.e., $O(\epsilon)$) $L^2$-norm. 
\end{Remark}

\begin{Remark}[Error estimates in weighted $L^2$-norm]\label{rem:errweightedl2}
	In view of the Ritz projection estimate \eqref{eq:ritzweightedl2}, one might aim for error estimates in the weighted $L^2$-norm with the hope to avoid the critical terms $(a_\epsilon)^{-1/2}\partial_t^iu_\epsilon$ discussed in the previous remark.
	Unfortunately, we are not aware of standard energy techniques that allow for such estimates. In the stationary case, one typically uses dual problems and the Aubin-Nitsche trick, but this also does not appear to be possible in the wave setting to the best of our knowledge.
	
	Therefore, we can only conjecture the following estimate
	\[\|\sqrt{a_\epsilon}(u-u_{H,m}^{ms})\|_{L^\infty(L^2)}\lesssim_T\|\sqrt{a_\epsilon}(1-\pi_h)\partial_tu_\epsilon\|_{L^1(L^2)}+(H+m^{d/2}\gamma^m)\|\sqrt{a_\epsilon}\nabla\partial_tu_\epsilon\|_{L^1(L^2)}.\]
	By differentiating \eqref{eq:wave_eq_variational_form} with respect to time, we can get the time regularity estimate
	\[\|\sqrt{a_\epsilon}\nabla\partial_tu_\epsilon\|_{L^1(L^2)}\lesssim_T\|\partial_tf\|_{L^2(L^2)}+\|v^0\|_{H^1(\Omega)}+\|w_2\|_{L^2(\Omega)}\]
	with $w_2=f(\cdot, 0)+\nabla\cdot (a_\epsilon\nabla u^0)$ as in the previous remark. If $\|w_2||_{L^2(\Omega)}$ is bounded independent of $\epsilon$, which is for instance the case for $u^0=0$ and $f$ independent of $\epsilon$, this implies an $\epsilon$-independent error estimate in the weighted $L^2$-norm as conjecture.
\end{Remark}

\begin{proof}[Proof of Theorem \ref{thm:error}]
\emph{Proof of \eqref{eq:errorenergy}:} Following \cite{LarTho}, we split $u-u_{H,m}^{ms}=\rho+\eta$ as already discussed. We directly obtain that
\begin{equation}\label{eq:errorideta}
(\partial_{tt}\eta, \psi)+b_\epsilon(\theta, \psi)=-(\partial_{tt}\rho, \psi)\quad \forall \psi\in V_{H,m}^{ms}.
\end{equation}
Testing with $\psi=\partial_t\eta$, we deduce
\[\frac{d}{dt}(\|\partial_t \eta\|_{L^2(\Omega)}^2+\|\sqrt{a_\epsilon}\nabla \eta\|_{L^2(\Omega)}^2)\leq \|\partial_{tt}\rho\|_{L^2(\Omega)}^2+\|\partial_t\eta\|_{L^2(\Omega)}^2.\]
Application of Gronwall's inequality yields
\[\|\partial_t\eta\|_{L^\infty(L^2)}+\|\sqrt{a_\epsilon}\nabla \eta\|_{L^\infty(L^2)}\lesssim_T\|\partial_t \eta(\cdot, 0)\|_{L^2(\Omega)}+\|\sqrt{a_\epsilon}\nabla \eta(\cdot, 0)\|_{L^2(\Omega)}+\|\partial_{tt}\rho\|_{L^1(0,T;L^2(\Omega))}.\]
The term $\|\partial_{tt}\rho\|_{L^1(0,T;L^2(\Omega))}$ can be treated with \eqref{eq:ritzl2}.
By the choice of the initial values, we have $\eta(\cdot, 0)=0$ and $\partial_t \eta(\cdot, 0)=0$. Combination with \eqref{eq:ritzenergy}--\eqref{eq:ritzl2} finishes the proof of \eqref{eq:errorenergy}.

\emph{Proof of \eqref{eq:errorl2}:} As in \cite{Bak1976}, we re-write \eqref{eq:errorideta} as
\[\frac{d}{dt}(\partial_t \eta, \psi)-(\partial_t \eta, \partial_t \psi)+b_\epsilon(\eta, \psi)=-\frac{d}{dt}(\partial_{t}\rho, \psi)+(\partial_t\rho, \partial_t\psi)\qquad \forall \psi\in L^1(0,T; V_{H,m}^{ms}).\]
Fixing some $0<t_0<T$, we choose $\psi(\cdot, t)=\int_t^{t_0}\eta(\cdot, s)\,ds$ and observe that $\partial_t\psi=-\eta$. This gives
\begin{align*}
\frac 12\frac{d}{dt}\|\eta(\cdot, t)\|_{L^2(\Omega)}^2-\frac12\frac{d}{dt}b_\epsilon\Bigl(\int_t^{t_0}\eta, \int_t^{t_0}\eta\Bigr)=-\frac{d}{dt}\Bigl(\partial_t(u-u_{H,m}^{ms}), \int_t^{t_0}\eta\Bigr)-(\partial_t\rho, \eta).
\end{align*}
Integration from $0$ to $t_0$ yields
\begin{align*}
\|\eta(\cdot, t_0)\|^2_{L^2(\Omega)}&\leq \|\eta(\cdot, 0)\|^2_{L^2(\Omega)}+\Bigl(\partial_t(u-u_{H,m}^{ms})(\cdot, 0), \int_0^{t_0}\eta\Bigr)-\int_0^{t_0}(\partial_t\rho, \eta)\\
&\leq \|\eta(\cdot, 0)\|^2_{L^2(\Omega)}+2\|\partial_t\rho\|^2_{L^1(L^2)}+\frac12\|\eta\|^2_{L^\infty(L^2)}.
\end{align*}
In the second step we used $\partial_t(u-u_{H,m}^{ms})(\cdot, 0)=v^0-P_{H,m}^{ms}v^0$ and $\eta\in V_{H,m}^{ms}$.
Taking the maximum over $0<t_0<T$, we arrive at
\begin{align*}
\|\eta\|^2_{L^\infty(L^2)}\lesssim \|\eta(\cdot, 0)\|^2_{L^2(\Omega)}+\|\partial_t \rho\|^2_{L^1(L^2)}
\end{align*}
The last term can be estimated with \eqref{eq:ritzl2} and, as in the previous step, we have $\eta(\cdot, 0)=0$.
\end{proof}

\subsection{Numerical illustrations}\label{subsec:numeric:numexp}
In this section, we illustrate the theoretical convergence rates with numerical experiments in one and two dimensions. The implementation is based on the \texttt{python}-module \texttt{gridlod} \cite{HeKe} and the Petrov-Galerkin formulation of the LOD method.
We always choose $\Omega=(0,1)^d$ and $T=0.25$ as well as $u^0=v^0=0$. We consider different right-hand sides, which are specified in each numerical experiment.
This set-up of zero initial values seems attractive from the viewpoint of ``compatible'' data  and time regularity estimates as discussed in Remark \ref{rem:timeregest}. Furthermore, it ensures that we are not in the ``simple'' case of Proposition \ref{prop:homhighcontrast1d} in the one-dimensional case.
In our experiments, we will consider periodic high-contrast $a_\epsilon$ as in the numerical simulations of Section \ref{subsec:asymptotic:numexp} and a random checkerboard realization (with high contrast) for $a_\epsilon$. The latter is generated as follows: On a subdivision of $\Omega$ into cuboids of side length $\epsilon$, $a_\epsilon$ is piecewise constant and the value of $a_\epsilon$ in each cuboid is chosen as $1$ or $\epsilon^2$ with probability $0.5$ (independently for all cuboids).
All the plots below show $L^\infty(L^2)$ or $L^\infty(L^2_a)$-errors between a reference solution and the time-discretized version of $u_{H,m}\in V_{H,m}^{ms}$, whose vector of degrees of freedom at time $t^n$ is given by $\zeta_m^n$ from \eqref{eq:wave_eq_space_time_discrete_LOD}, using the implicit midpoint rule. Here, $L^2_a$ denotes the usual $L^2$-space but with the $a_\epsilon$-weighted norm $\|\sqrt{a_\epsilon}\cdot\|_{L^2(\Omega)}$.

\paragraph*{One-dimensional case.}
We choose $\epsilon=2^{-10}$ and compute a reference solution using the mesh size $h=2^{-13}$, which clearly resolves the variations of $a_\epsilon$, and a time step $\tau=2^{-9}$.
First, we consider the periodic set-up with $f=x(x-1)\in H^1_0(\Omega)$. For the series of LOD discretization parameters $k=1,2,3$ and $H=2^{-2}, \ldots 2^{-7}$, Figure \ref{fig:1dperiodic} shows the errors in the two discussed norms, namely $L^\infty(L^2)$ and $L^\infty(L^2_a)$. Further, we compare the use of the standard interpolation operator $I_H$ and its weighted version (cf.~the discussion in Section \ref{subsec:numeric:lod}).
Except for the case $k=1$, we observe a quadratic convergence rate in both norms and irrespective of the chosen interpolation operator. This higher rate in comparison to Theorem \ref{thm:error} and (conjectured) Remark \ref{rem:errweightedl2} is related to higher spatial regularity.
While the convergence rates are roughly similar for both norms, the errors in the weighted $L^2$-norm are clearly smaller and, additionally, the convergence does not slow down for finer mesh sizes $H$. The latter effect may be related to a stronger $\epsilon$-dependency in the $L^2$-error estimate as discussed in Remarks \ref{rem:timeregest} and \ref{rem:errweightedl2}.

\begin{figure}
	\begin{subfigure}{0.5\textwidth}
		\includegraphics[scale=0.4, trim=0mm 0mm 15mm 10mm, clip=true]{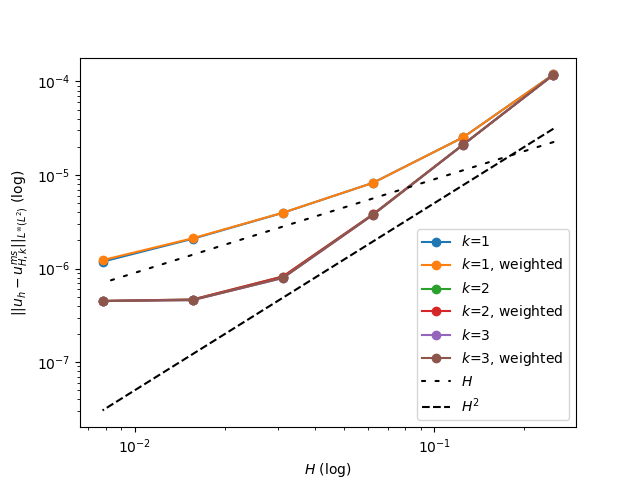}
		\captionsetup{justification=centering}
		\caption{Error in unweighted norm}
	\end{subfigure}
	\begin{subfigure}{0.5\textwidth}
		\includegraphics[scale=0.4, trim=0mm 0mm 15mm 10mm, clip=true]{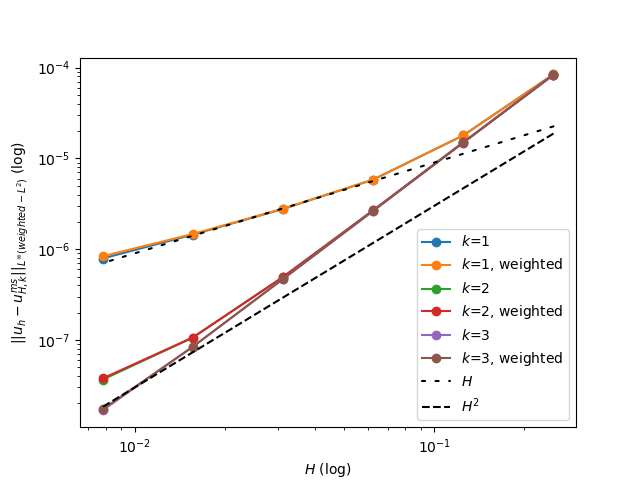}
		\captionsetup{justification=centering}
		\caption{Error in weighted norm}
	\end{subfigure}
	\caption{Error convergence for the one-dimensional periodic setting with $f=x(x-1)$}
	\label{fig:1dperiodic}
\end{figure}

Next, we study the random realization of $a_\epsilon$. As expected or conjectured, we observe linear convergence in both weighted and unweighted $L^2$-norms, see Figure \ref{fig:1drandom}. This underlines the applicability of the method beyond the periodic case. 
Note that for both choices of $a_\epsilon$, the choice of $I_H$ does not have a great influence on the observed errors in the target regime of large $H$.

\begin{figure}
	\begin{subfigure}{0.5\textwidth}
		\includegraphics[scale=0.4,trim=0mm 0mm 10mm 10mm, clip=true]{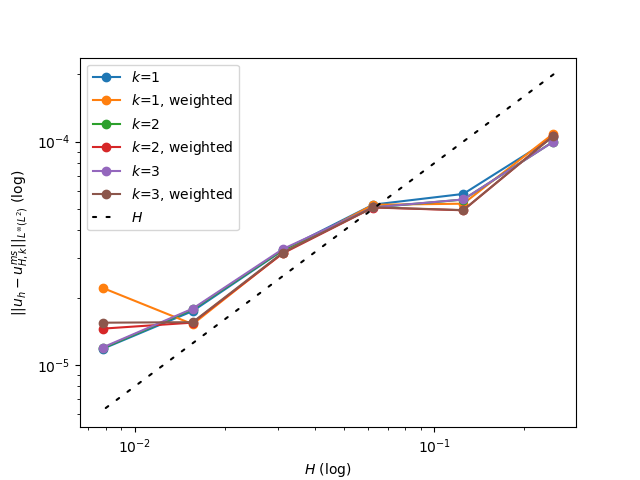}
		\captionsetup{justification=centering}
		\caption{Error in unweighted norm}
	\end{subfigure}
	\begin{subfigure}{0.5\textwidth}
		\includegraphics[scale=0.4, trim=0mm 0mm 10mm 10mm, clip=true]{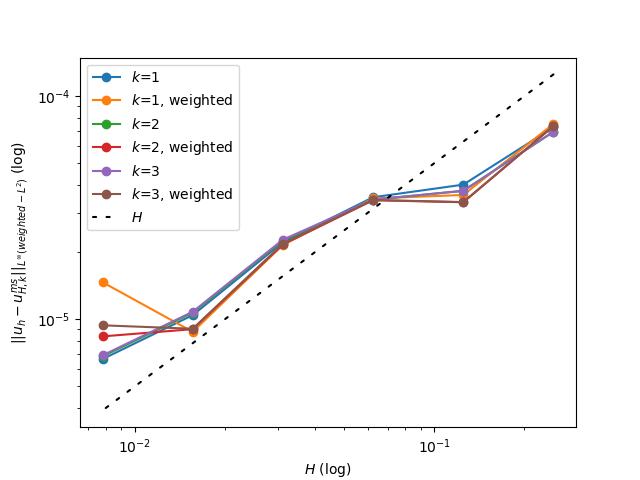}
		\captionsetup{justification=centering}
		\caption{Error in weighted norm}
	\end{subfigure}
	\caption{Error convergence for the one-dimensional random setting with $f=x(x-1)$}
	\label{fig:1drandom}
\end{figure}

Finally, we briefly consider  $f\equiv 1$. Discussions in \cite{MaPe2021} for the stationary problem indicate that regularity of $f$ may lead to higher convergence rates only if $f$ also satisfies the zero boundary conditions.
We focus on the case $k=3$ and the weighted $L^2$-norm. Since we previously did not observe an impact of the interpolation operator, we only depict results for the unweighted interpolation operator, but we confirmed our observations also for the weighted variant of $I_H$.
Figure \ref{fig:1d-f1} shows a similar convergence rate for the periodic and the random case, which lies somewhere between $0.5$ and the (conjectured) $1$. The reduction in the convergence rate possibly is due to less compatible data and therefore less regularity and/or more critical $\epsilon$ factors in the estimates.

\begin{figure}
	\centering
	\includegraphics[scale=0.4, trim=0mm 0mm 10mm 10mm, clip=true]{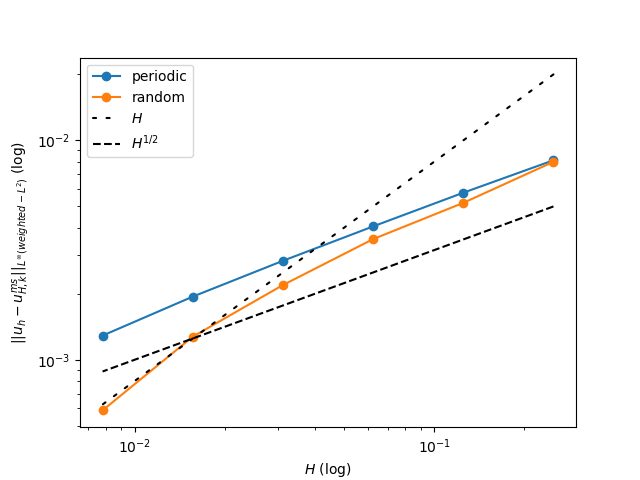}
	\captionsetup{justification=centering}
	\caption{Error convergence for the one-dimensional setting with $f=1$}
	\label{fig:1d-f1}
\end{figure}

\paragraph*{Two-dimensional case.}
In the two-dimensional setting, we choose $\epsilon =2^{-5}$ and compute the reference solution with mesh size $h=2^{-7}$ and $\tau=2^{-9}$.
The LOD discretization parameters are varied as $k=1,2,3$ and $H=2^{-2},\ldots, 2^{-5}$.
Results for the periodic setting with $f=x_1x_2(x_1-1)(x_2-1)\in H^1_0(\Omega)$ are depicted in Figure \ref{fig:2dperiodic}.  
In the two-dimensional setting, the different norms also show different error behavior. More than linear convergence seems to be achievable for the weighted $L^2$-norm overall, but at most linear convergence is observed in the standard $L^2$-norm.
The choice of the interpolation operator now has a larger impact and, interestingly, the unweighted variant gives the better results in this experiment.

\begin{figure}
	\begin{subfigure}{0.5\textwidth}
		\includegraphics[scale=0.4, trim=0mm 0mm 10mm 10mm, clip=true]{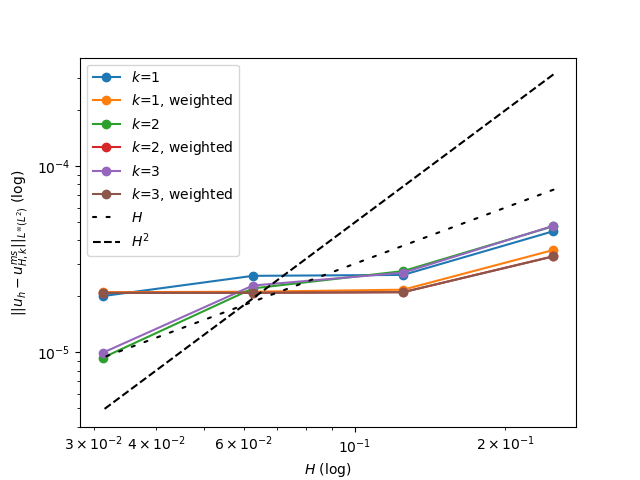}
		\captionsetup{justification=centering}
		\caption{Error in unweighted norm}
	\end{subfigure}
	\begin{subfigure}{0.5\textwidth}
		\includegraphics[scale=0.4, trim=0mm 0mm 10mm 10mm, clip=true]{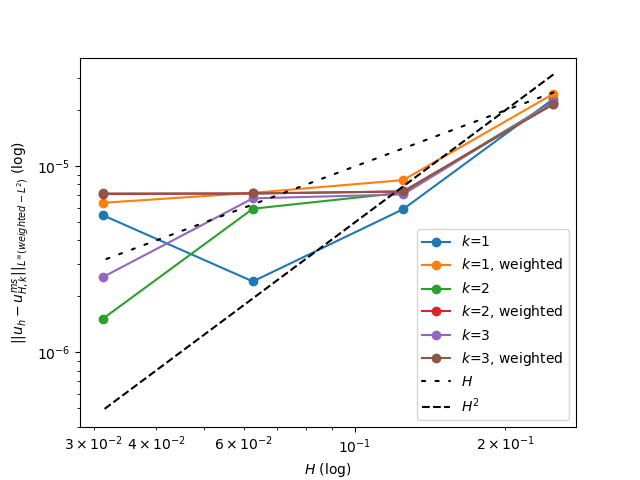}
		\captionsetup{justification=centering}
		\caption{Error in weighted norm}
	\end{subfigure}
	\caption{Error convergence for the two-dimensional periodic setting}
	\label{fig:2dperiodic}
\end{figure}

As most demanding test case, we consider the random checkerboard in two dimensions. We set $a_\epsilon$ as described in the beginning, but only in the inner domain $(0.25, 0.75)^2$, outside we set it equal to $1$. This set-up is to prevent small values of $a_\epsilon$ close to the boundary.
In the light of our discussion of the $\epsilon$-dependency of our error estimates, we choose $f$ in such a way that its support does not intersect with $\Omega_\epsilon$. Here, we set $f=1$ in $\Omega\setminus (0.25, 0.75)^2$ and zero elsewhere.
Despite the low regularity of $f$, the errors in Figure \ref{fig:2drandom} converge nicely, in particular we observe the conjectured linear convergence in the weighted $L^2$-norm. We omitted the results for $k=1$ here, since the localization/truncation error is dominating and the errors therefore do not decay.

\begin{figure}
	\begin{subfigure}{0.5\textwidth}
		\includegraphics[scale=0.4, trim=0mm 0mm 10mm 10mm, clip=true]{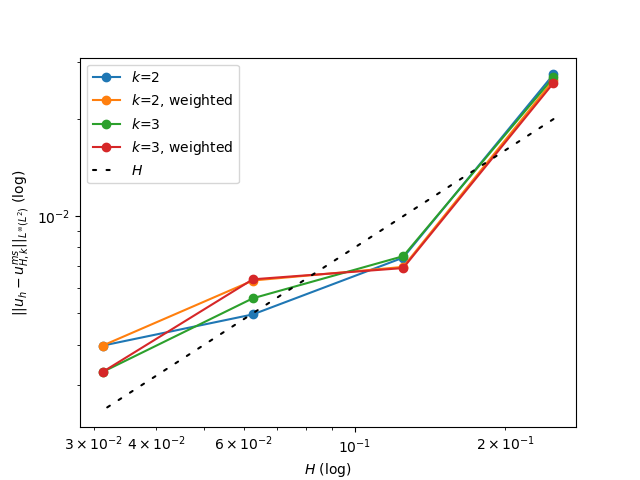}
		\captionsetup{justification=centering}
		\caption{Error in unweighted norm}
	\end{subfigure}
	\begin{subfigure}{0.5\textwidth}
		\includegraphics[scale=0.4, trim=0mm 0mm 10mm 10mm, clip=true]{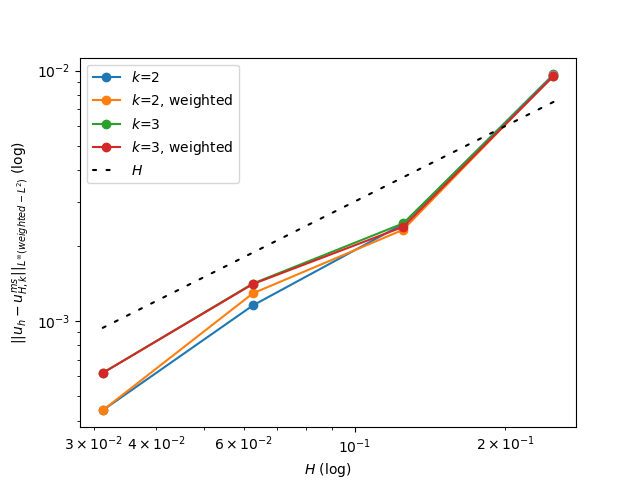}
		\captionsetup{justification=centering}
		\caption{Error in weighted norm}
	\end{subfigure}
	\caption{Error convergence for the two-dimensional random setting}
	\label{fig:2drandom}
\end{figure}

Overall, the two-dimensional results show expected or conjectured convergence rates, especially for the weighted $L^2$-norm, which seems to be the best choice in the high-contrast setting. Moreover, our numerical experiments underline the practicability of the LOD method even for very unstructured high-contrast coefficients. 
The influence of the interpolation operator needs to be further investigated in the high-contrast case -- both theoretically as well as numerically.

\section*{Conclusion}
We discussed wave propagation in high-contrast media, where the lower bound of the material coefficient is coupled to the fine-scale parameter $\epsilon$. Asymptotic analysis results in the periodic case underline the difference between low- and high-contrast settings. In particular, we could show a new result in one space dimension, where the limit solution takes a simple form. Numerical simulations illustrated these findings and the special role of the one-dimensional setting.
To cope with more general high contrast coefficients, we reviewed the Localized Orthogonal Decomposition method for the wave equation. We showed rigorous a priori error estimates in the semi-discrete setting and especially discussed the dependence on the high contrast. Several numerical experiments illustrated our findings and showed the potential of the method.
Open questions for future work concern (i) the error estimate in a weighted $L^2$-norm, which seems to have practical relevance in view of our numerical tests, and (ii) the selection of the interpolation operator. In the periodic setting, a better understanding of the homogenization limit and its numerical tractability in higher dimensions is desirable as well.
Overall our results clearly indicate the need for a better understanding of wave propagation in high-contrast media and that a lot of interesting effects are to be expected in the future. For instance, we did not search for resonance phenomena similar to those in the time-harmonic regime \cite{PeVe2020}.

\bibliographystyle{abbrv}
\bibliography{references}

\end{document}